\documentclass{amsart}
\usepackage{amsthm}
\usepackage{amsmath}
\usepackage{amssymb}
\usepackage{mathtools}
\usepackage{tikz-cd}
\usepackage{microtype}
\usepackage{calrsfs}
\DeclareMathAlphabet{\pazocal}{OMS}{zplm}{m}{n}
\usepackage{mathabx}
\usepackage{amsmath}
\usepackage[all]{xy}
\usepackage{undertilde}
\usepackage[numbers]{natbib}
\theoremstyle{plain}
\newtheorem*{theorem*}{Theorem}
\newtheorem{theo}{Theorem}[section]
\newtheorem{cor}[theo]{Corollary}
\newtheorem{defin}[theo]{Definition}
\newtheorem{prop}[theo]{Proposition}

\newtheorem{rem}[theo]{Remark}

\newtheorem*{lm*}{Lemma}
\newtheorem*{theo*}{Theorem}
\newtheorem*{cor*}{Corollary}
\newtheorem*{prop*}{Proposition}
\newtheorem*{def*}{Definition}
\newtheorem*{rem*}{Remark}

\begin{document}
	\title{Stability of projectively coresolved Gorenstein flat modules}

	\author{Dimitra-Dionysia Stergiopoulou}
	\address{Department of Mathematics, University of Athens, Panepistimioupolis, Athens 15784, Greece}
	\email{dstergiop@math.uoa.gr}

	\keywords{Gorenstein homological algebra, Gorenstein flat modules, Gorenstein dimension, stability}
	\subjclass{Primary: 16E05, 16E10, 16E30, 18G20}
	\begin{abstract}
		The stability of the class of projectively coresolved Gorenstein flat modules, under the very Gorenstein process used to define them, is proven in this paper. Moreover, a new characterization of the projectively coresolved Gorenstein flat dimension is given.
	\end{abstract}

	\maketitle
	
	\section {Introduction}	In the study of Gorenstein homological algebra, projective, injective and flat modules, on which classical homological algebra is based, are replaced with Gorenstein projective, Gorenstein injective and Gorenstein flat modules, respectively. Holm's metatheorem \cite{Hoo} states that every result in classical homological algebra has a counterpart in Gorenstein homological algebra. An open problem of the area is whether every Gorenstein projective module is Gorenstein flat. Saroch and Stovicek \cite{SS} introduced the class of projectively coresolved Gorenstein flat modules. These modules have nice properties; for instance they are simultaneously Gorenstein projective and Gorenstein flat \cite[Theorem 4.4]{SS}. Thus, projectively coresolved Gorenstein flat modules could probably play the role of projective modules in Gorenstein homological algebra. 
	
	Sather-Wagstaff, Sharif and White \cite{SSW} investigated the modules that arise from the iteration of the very procedure that leads to the Gorenstein projective modules and proved the stability of the classes of Gorenstein projective and Gorenstein injective modules. Using a special class of modules related to strongly Gorenstein flat modules, Bouchiba and Khaloui \cite{BK} proved the analogous stability of the class of Gorenstein flat modules. A natural question is whether the class of projectively coresolved Gorenstein flat modules is stable under this very Gorenstein process.

	In the first part of this paper, we prove the stability of the class of projectively coresolved Gorenstein flat modules (PGF modules, for short), which is the main result of Section 3. In particular, for every exact sequence of PGF modules $\textbf{G}=  \cdots \rightarrow G_1 \rightarrow G_0 \rightarrow G^0 \rightarrow G^1 \rightarrow \cdots$, such that $M\cong \textrm{Im}(G_0 \rightarrow G^0)$ and $H \otimes -$ preserves exactness of $\textbf{G}$ for every Gorenstein injective module $H$, the module $M$ is PGF (see Theorem \ref{finale}). A central role in the proof is played by the subcategory consisting of the $R$-modules $M$ for which there exists a short exact sequence of the form $0\rightarrow M \rightarrow G \rightarrow M \rightarrow 0$, where $G$ is a PGF module such that $I \otimes -$ preserves exactness of this sequence for every injective module $I$. A related class, which we call the class of strongly $n$-projectively coresolved Gorenstein flat modules (strongly $n$-PGF modules, for short), where $n$ is a nonnegative integer, is studied in the second part of the paper. Using this class we give a new characterization of the projectively coresolved Gorenstein flat dimension, defined by Dalezios and Emmanouil \cite{DE}. Strongly $0$-PGF modules are used in \cite{SS} to prove that every PGF module is Gorenstein projective. Strongly $n$-Gorenstein projective, injective and flat modules have been studied in \cite{BM} for $n=0$ and in \cite{MT} for general $n$.

	\section{Preliminaries}
	In this section, we collect certain notions and preliminary results that will be used in the sequel. Throughout this paper, $R$ is a unital associative ring and all modules are left $R$-modules. 

Saroch and Stovicek \cite{SS} defined the notion of PGF modules.  

\begin{defin}\label{definpgf}
	An R-module $M$ is called PGF, if there exists an exact sequence of projective modules 
	$ \textbf{P}= \cdots \rightarrow P_1 \rightarrow P_0 \rightarrow P^0 \rightarrow P^1 \rightarrow \cdots ,$ such that $M\cong \textrm{Im}(P_0 \rightarrow P^0)$ and such that $I\otimes -$ preserves the exactness of $\textbf{P}$ whenever I is an injective module.
\end{defin}
The class of PGF $R$-modules, denoted by ${\tt PGF}(R)$, is closed under extensions, direct sums, direct summands and kernels of epimorphisms \cite{SS}.

The following proposition gives a characterization of PGF modules.

\begin{prop}\label{pgf}
	For every module M, the following are equivalent:
	\begin{enumerate}
		\item M is PGF.
		\item M satisfies the following two conditions:
		\begin{itemize}
		\item[(i)] There exists an exact sequence of the form $0\rightarrow M \rightarrow P^0 \rightarrow P^1 \rightarrow \cdots$, where each $P^i$ is projective, such that $I \otimes -$ preserves exactness of this sequence for every injective module $I$.
		\item[(ii)] $\textrm{Tor}_i^R(I,M)=0$ for every $i>0$ and every injective module $I$.
	\end{itemize}
\item There exists a short exact sequence of the form $0\rightarrow M \rightarrow P \rightarrow G \rightarrow 0$, where $P$ is projective and $G$ is a PGF module.
	\end{enumerate}
\end{prop}

\begin{proof} Using the definition of  PGF modules, the equivalence $(1)\Leftrightarrow (2)$ is obtained by standard arguments. Also, by definition, the implication $(1)\Rightarrow (3)$ is clear. 
$(3)\Rightarrow (2)$: Let $0\rightarrow M \rightarrow P \rightarrow G \rightarrow 0$ be an exact sequence, where $P$ is projective and $G$ is a PGF module. Since $G$ is PGF, the implication $(1)\Rightarrow (2)$ yields $\textrm{Tor}_i^R(I,G)=0$ for every $i>0$ and every injective module $I$. Let $I$ be an injective module. Then, the short exact sequence $0\rightarrow M \rightarrow P \rightarrow G \rightarrow 0$ induces a long exact sequence of the form $\cdots \rightarrow \textrm{Tor}_{i+1}^R(I,G) \rightarrow \textrm{Tor}_{i}^R(I,M)\rightarrow \textrm{Tor}_{i}^R(I,P)\rightarrow \cdots ,$ where $i>0$, which implies that $\textrm{Tor}_{i}^R(I,M)=0$ for every $i>0$. Moreover, there exists an exact sequence of the form $0\rightarrow G \rightarrow P^0 \rightarrow  P^1 \rightarrow \cdots $, such that $I\otimes -$ preserves the exactness of this sequence. We obtain an exact sequence of the form $0 \rightarrow M\rightarrow P\rightarrow P^0 \rightarrow  P^1 \rightarrow \cdots ,$ such that $I\otimes -$ preserves exactness of this sequence for every injective module $I$.\end{proof}

\section{Stability of PGF modules}
In this section we prove that the class of PGF modules is stable under the very Gorenstein process used to define PGF modules. 

Throughout this section we use the following notation.

We denote by ${\tt PGF}^{(2)}(R)$ (respectively, ${\tt PGF}^{(2)}_{\pazocal{I}}(R)$) the subcategory of the R-modules $M$ for which there exists an exact sequence of PGF modules $$\textbf{G}= \cdots \rightarrow G_1 \rightarrow G_0 \rightarrow G^0 \rightarrow G^1 \rightarrow \cdots,$$ such that $M\cong \textrm{Im}(G_0 \rightarrow G^0)$ and such that $H\otimes -$ (respectively, $I\otimes -$) preserves the exactness of $\textbf{G}$ whenever $H$ is a Gorenstein injective module (respectively, $I$ is an injective module).

Since every injective module is also Gorenstein injective, we have the inclusions ${\tt {PGF}}(R)\subseteq {\tt {PGF}}^{(2)}(R) \subseteq {\tt {PGF}}^{(2)}_{\pazocal{I}}(R).$

Also, we denote by ${\tt {SPGF}}^{(2)}_{\pazocal{I}}(R)$ the subcategory of the R-modules $M$ for which there exists a short exact sequence of the form $0\rightarrow M \rightarrow G \rightarrow M \rightarrow 0$, where $G$ is a PGF module, such that $I\otimes -$ preserves the exactness of this sequence whenever $I$ is an injective module.

\begin{prop}\label{prop44}
	Let $M\in {\tt{PGF}}^{(2)}_{\pazocal{I}}(R)$. Then, $\textrm{Tor}_i^R(I',M)=0$ for every $i>0$ and every module $I'$ with finite injective dimension.
\end{prop}

\begin{proof} Let $M\in {\tt PGF}^{(2)}_{\pazocal{I}}(R)$ and $I'$ be a module with finite injective dimension. Then, there exists an exact sequence of PGF modules $\cdots \rightarrow G_1 \rightarrow G_0 \rightarrow G^0 \rightarrow G^1 \rightarrow \cdots,$ where $M\cong \textrm{Im}(G_0 \rightarrow G^0)$ and the functor $I\otimes - $ preserves the exactness of this sequence for every injective module $I$. Let $I'$ be a module with $\textrm{id}_R(I')=n<\infty$. We proceed by induction on $n\geq0$. Consider the short exact sequence $0\rightarrow K \rightarrow G_0 \rightarrow M \rightarrow 0$, where $K=\textrm{Im}(G_1\rightarrow G_0)\in {\tt PGF}^{(2)}_{\pazocal{I}}(R)$. Then, for every injective module $I'$ we have a short exact sequence of the form $0\rightarrow I'\otimes K \rightarrow I'\otimes G_0 \rightarrow I'\otimes M \rightarrow 0$. Since the module $G_0$ is PGF, Proposition \ref{pgf} (2) implies that $\textrm{Tor}_i^R(I',G_0)=0$ for every $i>0$. Then, the short exact sequence $0\rightarrow K \rightarrow G_0 \rightarrow M \rightarrow 0$ induces a long exact sequence of the form $$\cdots \rightarrow \textrm{Tor}_1^R(I',G_0) \rightarrow \textrm{Tor}_1^R(I',M)\rightarrow I'\otimes K \rightarrow I'\otimes G_0 \rightarrow I'\otimes M \rightarrow 0, $$ which implies that $\textrm{Tor}_1^R(I',M)=0$. 
	Moreover, the long exact sequence $$\cdots\rightarrow \textrm{Tor}_{i+1}^R(I',G_0)\rightarrow \textrm{Tor}_{i+1}^R(I',M) \rightarrow \textrm{Tor}_{i}^R(I',K)\rightarrow \textrm{Tor}_{i}^R(I',G_0) \rightarrow \cdots,$$ where $i>0$, yields $\textrm{Tor}_{i+1}^R(I',M) = \textrm{Tor}_{i}^R(I',K)$ for every $i>0$. Thus, using induction on $i$ and the fact that $K$ lies in ${\tt PGF}^{(2)}_{\pazocal{I}}(R)$, we obtain that $\textrm{Tor}_i^R(I',M)=0$ for every $i>0$ and every injective module $I'$.
	
	We suppose now that $n\geq 1$. Let $0\rightarrow I' \rightarrow I_0 \rightarrow I_1 \rightarrow \cdots \rightarrow I_n \rightarrow 0$ be an injective resolution of $I'$ of length $n$ and $J=\textrm{Im}(I_0 \rightarrow I_1)$. Since $\textrm{id}_R (J)\leq n-1$, our inductive hypothesis implies that $\textrm{Tor}_i^R(J,M)=0$ for every $i>0$. Thus, the short exact sequence $0\rightarrow I' \rightarrow I_0 \rightarrow J \rightarrow 0$ induces a long exact sequence of the form $ \cdots \rightarrow \textrm{Tor}_{i+1}^R(J,M)\rightarrow \textrm{Tor}_{i}^R(I',M) \rightarrow \textrm{Tor}_{i}^R(I_0,M) \rightarrow \cdots,$ where $i>0$, from which we obtain that $\textrm{Tor}_{i}^R(I',M)=0$ for every $i>0$. \end{proof}

The following proposition gives a characterization of the subcategory ${\tt SPGF}^{(2)}_{\pazocal{I}}(R)$.

\begin{prop}\label{props2pgfi}
	For every module M, the following are equivalent:
	\begin{enumerate}
		\item $M \in {\tt SPGF}^{(2)}_{\pazocal{I}}(R)$.
		\item There exists a short exact sequence $0\rightarrow M \rightarrow G \rightarrow M \rightarrow 0$ such that G is a PGF module and $\textrm{Tor}_1^R(I,M)=0$ for every injective module I.
		\item There exists a short exact sequence $0\rightarrow M \rightarrow G \rightarrow M \rightarrow 0$ such that G is a PGF module and $\textrm{Tor}_i^R(I,M)=0$ for every $i>0$ and every injective module I.
		\item There exists a short exact sequence $0\rightarrow M \rightarrow G \rightarrow M \rightarrow 0$ such that G is a PGF module and $\textrm{Tor}_i^R(I',M)=0$ for every $i>0$ and every module $I'$ with finite injective dimension.
		\item There exists a short exact sequence $0\rightarrow M \rightarrow G \rightarrow M \rightarrow 0$ such that G is a PGF module and $I' \otimes -$ preserves exactness of this sequence for every module $I'$ with finite injective dimension.
		
	\end{enumerate}
\end{prop}

\begin{proof}
	This follows immediately from the definition of the class ${\tt SPGF}^{(2)}_{\pazocal{I}}(R)$ and Proposition \ref{prop44}, using standard arguments. 
\end{proof}

\begin{prop}\label{final3}
	Every module in ${\tt PGF}^{(2)}_{\pazocal{I}}(R)$ is a direct summand of a module in ${\tt SPGF}^{(2)}_{\pazocal{I}}(R)$.
\end{prop}

\begin{proof}
	Let $M$ be a module in ${\tt PGF}^{(2)}_{\pazocal{I}}(R)$. Then, there exists an exact sequence of PGF modules 
	$ \textbf{G}= \cdots \rightarrow G_1 \xrightarrow{d_1^G} G_0 \xrightarrow{d_0^G} G_{-1} \xrightarrow{d_{-1}^G} G_{-2} \rightarrow \cdots$,
	such that $M\cong \textrm{Im}(d_0^G)$ and such that the sequence $I\otimes \textbf{G}$ is exact for every injective module $I$. For every $n \in \mathbb{Z}$, we denote by $\Sigma^n\textbf{G}$ the exact sequence obtained from $\textbf{G}$ by increasing all indices by n: $(\Sigma^n\textbf{G})_i=G_{i-n}$ and $d_i^{\Sigma^n G}=d_{i-n}^G$ for every $i\in \mathbb{Z}$. Consider now the exact sequence
	$$\bigoplus_{n \in \mathbb{Z}}(\Sigma^n\textbf{G})= \cdots \rightarrow \bigoplus_{i \in \mathbb{Z}}G_i \xrightarrow{\bigoplus_{i \in \mathbb{Z}}d_i^G}\bigoplus_{i \in \mathbb{Z}}G_i \xrightarrow{\bigoplus_{i \in \mathbb{Z}}d_i^G}\bigoplus_{i \in \mathbb{Z}}G_i \rightarrow \cdots.$$ Since the class ${\tt PGF}(R)$ is closed under direct sums, we obtain that the module $\bigoplus_{i \in \mathbb{Z}}G_i$ is also PGF. Moreover, by Proposition 20.2 (3) of \cite{AF}, we have the isomorphism of complexes 
	$I\otimes (\bigoplus_{n \in \mathbb{Z}}(\Sigma^n\textbf{G})) \cong \bigoplus_{n \in \mathbb{Z}}(I\otimes \Sigma^n\textbf{G})$ which is an exact sequence for every injective module $I$. Thus, $\textrm{Im}(\bigoplus_{i \in \mathbb{Z}}d_i^G)$ lies in ${\tt SPGF}^{(2)}_{\pazocal{I}}(R)$ and $M$ is a direct summand of this module.\end{proof}

\begin{defin}
	Let $M$ be a module in ${\tt SPGF}^{(2)}_{\pazocal{I}}(R)$. We say that a module N is an $M$-type if there exists a short exact sequence of the form $0\rightarrow M \rightarrow N \rightarrow G \rightarrow 0,$ where $G$ is a PGF module. 
\end{defin}

\begin{prop}\label{final2}
	Let $M$ be a module in ${\tt SPGF}^{(2)}_{\pazocal{I}}(R)$ and $N$ be an $M$-type. Then, the following hold.
	\begin{enumerate}
		\item $\textrm{Tor}_i^R(I,N)=0$ for every $i>0$ and every injective module $I$.
		\item There exists an exact sequence of the form $0\rightarrow N \rightarrow P \rightarrow K \rightarrow 0$, where $P$ is projective, $K$ is an $M$-type module and the functor $I\otimes -$ preserves the exactness of this sequence for every injective module $I$.
	\end{enumerate}
	
\end{prop}

\begin{proof}
	(1) Let $I$ be an injective module. Since $N$ is an $M$-type, there exists a short exact sequence of the form $0\rightarrow M \rightarrow N \rightarrow G \rightarrow 0$, where $G$ is PGF, which induces a long exact sequence of the form $\cdots \rightarrow \textrm{Tor}_{i}^R(I,M) \rightarrow \textrm{Tor}_{i}^R(I,N) \rightarrow \textrm{Tor}_{i}^R(I,G) \rightarrow \cdots,$ where $i>0$. Since $M\in {\tt SPGF}^{(2)}_{\pazocal{I}}(R)$, by Proposition \ref{props2pgfi} (3), we have $\textrm{Tor}_{i}^R(I,M)=0$ for every $i>0$. Moreover, Proposition \ref{pgf} (2) yields $\textrm{Tor}_{i}^R(I,G)=0$ for every $i>0$. Consequently, $\textrm{Tor}_{i}^R(I,N)=0$, for every $i>0$ and every injective module $I$.
	
	(2) Since $M\in {\tt SPGF}^{(2)}_{\pazocal{I}}(R)$, there exists a short exact sequence of the form $0\rightarrow M \rightarrow G' \rightarrow M \rightarrow 0$, where $G'$ is a PGF module. Since $N$ is an $M$-type, there exists also a short exact sequence of the form $0\rightarrow M \rightarrow N \rightarrow G \rightarrow 0$, where $G$ is a PGF module. Consider the pushout diagram of the above short exact sequences:
	\[
	\begin{array}{ccccccccc}
		& & 0 & & 0 & & & &\\
		& & \downarrow & & \downarrow & & & &\\
		0&\rightarrow & M &\rightarrow & G' &\rightarrow & M & \rightarrow &0\\
		& & \downarrow & & \downarrow & & \parallel & &\\
		0&\rightarrow & N &\rightarrow & F &\rightarrow & M & \rightarrow &0\\
		& & \downarrow & & \downarrow & & & &\\
		& & G & = & G & & & &\\
		& & \downarrow & & \downarrow & & & &\\
		& & 0 & & 0 & & & &
	\end{array}
	\]
	\noindent Since the class ${\tt PGF}(R)$ is closed under extensions, using the short exact sequence $0\rightarrow G' \rightarrow F \rightarrow G \rightarrow 0$, we obtain that the module $F$ is also PGF. Thus, there exists a short exact sequence of the form $0\rightarrow F \rightarrow P \rightarrow F' \rightarrow 0$, where $P$ is projective and $F'$ is PGF. Consider now the following pushout diagram:
	\[
	\begin{array}{ccccccccc}
		& & & & 0 & & 0 & &\\
		& & & & \downarrow & & \downarrow & &\\
		0&\rightarrow & N &\rightarrow & F &\rightarrow & M & \rightarrow &0\\
		& & \parallel & & \downarrow & & \downarrow & &\\
		0&\rightarrow & N &\rightarrow & P &\rightarrow & K & \rightarrow &0\\
		& & & & \downarrow & & \downarrow & &\\
		& & & & F' & = & F' & &\\
		& & & & \downarrow & & \downarrow & &\\
		& & & & 0 & & 0 & &
	\end{array}
	\]
	\noindent Since $F'$ is PGF, the module $K$ is an $M$-type. By (1) we have $\textrm{Tor}_1^R(I,K)=0$ for every injective module $I$. Thus, the sequence $0\rightarrow I\otimes N \rightarrow I\otimes P \rightarrow I\otimes K\rightarrow 0$ is exact for every injective module $I$.\end{proof}

\begin{cor}\label{corfinal}
	Let $M$ be a module in ${\tt SPGF}^{(2)}_{\pazocal{I}}(R)$ and $N$ be an $M$-type module. Then, $N$ is PGF.
\end{cor}

\begin{proof}
	Since $N$ is an $M$-type module, Proposition \ref{final2} (2) implies that there exists a short exact sequence of the form $0\rightarrow N \rightarrow P^0 \rightarrow K \rightarrow 0$, where $P^0$ is projective, $K$ is an $M$-type and the functor $I\otimes -$ preserves the exactness of this sequence for every injective module $I$. The iteration of this process yields an exact sequence of the form $0\rightarrow N \rightarrow P^0 \rightarrow P^1 \rightarrow P^2 \rightarrow \cdots,$ where $P^i$ is projective for every $i\geq 0$ and the functor $I\otimes -$ preserves the exactness of this sequence for every injective module $I$. Using Proposition \ref{final2} (1), we also have $\textrm{Tor}_{i}^R(I,N)=0$ for every $i>0$ and every injective module $I$. Thus, Proposition \ref{pgf} (2) implies that $N$ is PGF.
\end{proof}

\begin{theo} \label{finale}
	${\tt PGF}(R)= {\tt PGF}^{(2)}(R).$
\end{theo}

\begin{proof}
	It suffices to prove that ${\tt PGF}^{(2)}_{\pazocal{I}}(R)\subseteq {\tt PGF}(R)$. Since the class ${\tt PGF}(R)$ is closed under direct summands, by Proposition \ref{final3} it suffices to prove that ${\tt SPGF}^{(2)}_{\pazocal{I}}(R)\subseteq {\tt PGF}(R)$. Let $M$ be a module in ${\tt SPGF}^{(2)}_{\pazocal{I}}(R)$. Letting $G=0$ in Definition 3.4, it follows that $M$ is an $M$-type. Thus, Corollary 3.6 implies that $M$ is PGF. \end{proof}

    \section{Strongly $n$-PGF modules}
	
	In this section, we define the notion of strongly $n$-PGF modules and we give a new characterization of modules with finite PGF-dimension. In particular, we prove that an $R$ module $M$ has PGF-dimension less or equal to $n$ if and only if it is a direct summand of a strongly $n$-PGF module (see Theorem \ref{theo1}).
	
	\begin{defin}
		An R-module M is called strongly projectively coresolved Gorenstein flat (strongly PGF for short), if there exists a short exact sequence of the form 
		$ 0\rightarrow M \rightarrow P \rightarrow M \rightarrow 0 $
		such that P is a projective R-module and $I\otimes -$ preserves exactness of this sequence whenever $I$ is an injective module.
	\end{defin}
\begin{rem}\label{rpd}
	Let $M$ be a PGF module. Corollary 4.5 of \cite{SS} yields $\textrm{Ext}^i_R(M,P)=0$ for every $i>0$ and every projective module $P$. Consequently, every strongly PGF module is also strongly Gorenstein projective and strongly Gorenstein flat (see Definitions 2.1 and 3.1 of \cite{BM}). Thus, ${\tt SPGF}(R)= {\tt SGProj}(R) \cap {\tt SGFlat}(R)$, where ${\tt SPGF}(R)$, ${\tt SGProj}(R)$ and ${\tt SGFlat}(R)$ are the classes of strongly PGF, strongly Gorenstein projective and strongly Gorenstein injective modules respectively. A schematic presentation is given below:
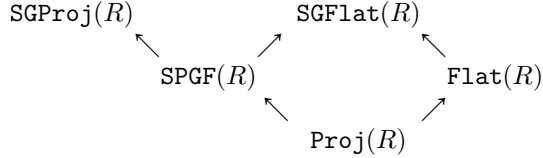
\[
\begin{array}{ccccccc}
	{\tt SGProj}(R) \!\!\! & & & & \!\!\! {\tt SGFlat}(R) \!\!\! & & \\
	& \!\!\! \nwarrow \!\!\! & & \!\!\! \nearrow& \!\!\! & \!\!\! 
	\nwarrow \!\!\! & \\
	& & \!\!\! {\tt SPGF}(R) \!\!\! & & & & \!\!\! {\tt Flat}(R) \!\!\! \\
	& & & \!\!\! \nwarrow \!\!\! & & \!\!\! \nearrow \!\!\! & \\
	& & & & \!\!\! {\tt Proj}(R) \!\!\! & & 
\end{array}
\]
	Here, ${\tt Proj(R)}$ and ${\tt Flat}(R)$ denote the classes of projective and flat modules, respectively, and all arrows are inclusions. Moreover, by \cite{SS} we have ${\tt Proj}(R)={\tt PGF}(R) \cap {\tt Flat}(R)={\tt SPGF}(R) \cap {\tt Flat}(R)$.
	\end{rem}
\begin{defin} \label{defns}
	An R-module M is called strongly n-projectively coresolved Gorenstein flat (strongly n-PGF for short), if there exists a short exact sequence of the form 
	$ 0 \rightarrow M\rightarrow F \rightarrow M \rightarrow 0 $, such that $\textrm{pd}_R(F)\leq n$ and $\textrm{Tor}_{n+1}^{\textrm{R}}(I,M)=0$ for every injective module $I$.
	
\end{defin}

\begin{rem*}
	A direct consequence of the above definition is that the strongly 0-PGF modules are precisely the strongly PGF modules.
\end{rem*}

\begin{cor}\label{corr}
	Every module with projective dimension less than or equal to $n$ is a strongly n-PGF module.
\end{cor}

\begin{proof}
	Let M be a module with $\textrm{pd}_R(M)\leq n$. Consider the short exact sequence $0 \rightarrow M \rightarrow M\oplus M \rightarrow M \rightarrow 0$, where $\textrm{pd}_R(M\oplus M)\leq n$. Thus, $\textrm{Tor}_{n+1}^R(I,M)=0$ for every injective module $I$ and $M$ is a strongly $n$-PGF module.
\end{proof}

The next result gives a simple characterization of strongly $n$-PGF modules.

\begin{prop}
	\label{n-spgf}
	For every module M, the following are equivalent:
	\begin{enumerate}
		\item M is strongly n-PGF.
		\item There exists a short exact sequence $0\rightarrow M \rightarrow F \rightarrow M \rightarrow 0$ such that $\textrm{pd}_R (F)\leq n$ and $\textrm{Tor}_i^R(I,M)=0$ for every $i>n$ and every injective module I.
		\item There exists a short exact sequence $0\rightarrow M \rightarrow F \rightarrow M \rightarrow 0$ such that $\textrm{pd}_R(F)\leq n$ and $\textrm{Tor}_i^R(I',M)=0$ for every $i>n$ and every module $I'$ with finite injective dimension.
		\item There exists a short exact sequence $0\rightarrow M \rightarrow F \rightarrow M \rightarrow 0$ such that $\textrm{pd}_R(F)\leq n$ and $\textrm{Tor}_{n+1}^R(I',M)=0$ for every module $I'$ with finite injective dimension.
		\end{enumerate}
\end{prop}
	
	\begin{proof} This follows immediately from the Definition of strongly $n$-PGF modules, using standard arguments. \end{proof}


	
 The PGF-dimension of a module $M$, denoted by $\textrm{PGF-dim}_R(M)$, is the minimal $n$ such that there exists an exact sequence of the form $0\rightarrow G_n \rightarrow G_{n-1} \rightarrow \cdots \rightarrow G_0 \rightarrow M \rightarrow 0,$ where $G_0, \dots ,G_{n-1}, G_n$ are PGF modules.
	
	Throughout the rest of this section, we use the following results concerning $\textrm{PGF-dim}_R(M)$.
	
	\begin{prop}\label{oplus}\rm{(\cite[Proposition 2.3]{DE})}
		Let $(M_i)_i$ be a family of modules and $M=\bigoplus_i M_i$ be their direct sum. Then, $\textrm{PGF-dim}_R(M)=\textrm{sup}_{i}\{ \textrm{PGF-dim}_R(M_i)\}$.
	\end{prop}

\begin{prop}\label{three}\rm{(\cite[Proposition 2.4]{DE})} Let $0\rightarrow M' \rightarrow M \rightarrow M'' \rightarrow 0$ be a short exact sequence of modules. Then,
	\begin{enumerate}
		\item $\textrm{PGF-dim}_R(M)\leq \textrm{max}\{\textrm{PGF-dim}_R(M'),\textrm{PGF-dim}_R(M'')\}$,
		\item $\textrm{PGF-dim}_R(M')\leq \textrm{max}\{\textrm{PGF-dim}_R(M),\textrm{PGF-dim}_R(M'')\}$,
		\item $\textrm{PGF-dim}_R(M'')\leq 1+\textrm{max}\{\textrm{PGF-dim}_R(M),\textrm{PGF-dim}_R(M')\}$.
	\end{enumerate}
\end{prop}

\begin{prop}\label{precover} \rm{(\cite[Theorem 3.4 (iii)]{DE})}
	Let $M$ be a module and $n$ a nonnegative integer. Then, $\textrm{PGF-dim}_R(M)=n\geq 0$ if and only if there exists a short exact sequence of the form $0 \rightarrow M \rightarrow D \rightarrow G \rightarrow 0$, where G is a PGF-module and $\textrm{pd}_R(D)=n$.
\end{prop}

	\begin{prop}\label{pd}\rm{(\cite[Corollary 3.7]{DE})}
		Let $M$ be a module such that $\textrm{pd}_R(M)<\infty$. Then, $\textrm{PGF-dim}_R(M)=\textrm{pd}_R(M)$.
	\end{prop}
	
	We continue with our results of this section.

	\begin{prop}\label{final}
		Let $n$ be a positive integer and $M$ be a strongly n-PGF module. Then, the following hold:
		\begin{enumerate}
			\item If $0\rightarrow N \rightarrow P_{n-1} \rightarrow \cdots \rightarrow P_0 \rightarrow M \rightarrow 0$ is an exact sequence where all $P_i$ are projective, then $N$ is strongly PGF and consequently $\textrm{PGF-dim}_R(M)\leq n$.
			\item Moreover, if $0\rightarrow M \rightarrow F \rightarrow M \rightarrow 0$ is a short exact sequence where $\textrm{pd}_R(F)<\infty$, then $\textrm{PGF-dim}_R(M)=\textrm{pd}(F)$ and consequently $M$ is strongly $k$-PGF with $k:=\textrm{pd}_R(F)$.
		\end{enumerate}
	\end{prop}

\begin{proof}
	$(1)$ Since $M$ is strongly $n$-PGF, there exists a short exact sequence of the form $0\rightarrow M \rightarrow F \rightarrow M \rightarrow 0$,  such that $\textrm{pd}_R(F)\leq n$ and $\textrm{Tor}_{n+1}^R(I,M)=0$ for every injective module $I$. Since the exact sequence $0\rightarrow N \rightarrow P_{n-1} \rightarrow \cdots \rightarrow P_0 \rightarrow M \rightarrow 0$ is a truncated projective resolution of $M$, there is a module $Q$ such that the following diagram is commutative:
\[
\begin{array}{ccccccccccccc}
	& & 0 & & 0 & & & & 0 & & 0 & &\\
    & &\downarrow & &\downarrow & & & &\downarrow & & \downarrow & & \\
    0 &\rightarrow & N & \rightarrow & P_{n-1} & \rightarrow & \cdots & \rightarrow & P_0 & \rightarrow & M & \rightarrow & 0 \\
    & &\downarrow & &\downarrow & & & &\downarrow & & \downarrow & & \\
    0 &\rightarrow & Q & \rightarrow & P_{n-1}\oplus P_{n-1} & \rightarrow & \cdots & \rightarrow & P_0 \oplus P_0 & \rightarrow & F & \rightarrow & 0 \\
    & &\downarrow & &\downarrow & & & &\downarrow & & \downarrow & & \\
    0 &\rightarrow & N & \rightarrow & P_{n-1} & \rightarrow & \cdots & \rightarrow & P_0 & \rightarrow & M & \rightarrow & 0 \\
    & &\downarrow & &\downarrow & & & &\downarrow & & \downarrow & & \\
	& & 0 & & 0 & & & & 0 & & 0 & &
\end{array}
\]
Since $\textrm{pd}_R(F)\leq n$, it follows that $Q$ is projective module. We observe that $\textrm{Tor}_1^R(I,N)=\textrm{Tor}_{n+1}^R(I,M)=0$ for every injective module $I$. Thus, Proposition \ref{n-spgf} (2) implies that the module $N$ is strongly PGF.

$(2)$ We consider a short exact sequence of the form $0\rightarrow M \rightarrow F \rightarrow M \rightarrow 0 $, such that $\textrm{pd}_R(F)=k<\infty$. Consider a truncated projective resolution of the module $M$ of length $n$,
$0\rightarrow N \rightarrow P_{n-1} \rightarrow \cdots \rightarrow P_0 \rightarrow M \rightarrow 0.$
Using $(1)$ which we have already proved, we obtain that $N$ is strongly PGF. By Proposition \ref{n-spgf} (2), there exists a short exact sequence $0\rightarrow N \rightarrow P \rightarrow N \rightarrow 0$ such that $P$ is projective and $\textrm{Tor}_i^R(I,N)=0$ for every $i>0$ and every injective module $I$. Let $I$ be an injective module. The short exact sequence $0\rightarrow M \rightarrow F \rightarrow M \rightarrow 0$ induces a long exact sequence of the form $$\cdots \rightarrow \textrm{Tor}_{i+1}^R(I,F) \rightarrow \textrm{Tor}_{i+1}^R(I,M)\rightarrow \textrm{Tor}_i^R(I,M)\rightarrow \textrm{Tor}_i^R(I,F) \rightarrow \cdots,$$ implying that  $\textrm{Tor}_{i+1}^R(I,M)=\textrm{Tor}_i^R(I,M)$ for every $i>k$. Thus, $\textrm{Tor}_{i}^R(I,M)=\textrm{Tor}_{n+i}^R(I,M)=\textrm{Tor}_{i}^R(I,N)=0$ for every $i>k$. By Proposition \ref{n-spgf} (2), we conclude that $M$ is strongly $k$-PGF. It remains to prove that $\textrm{PGF-dim}_R (M)=k$. By (1) we have $\textrm{PGF-dim}_R (M)\leq k$. Since $\textrm{pd}_R(F)=k<\infty$, Proposition \ref{pd} implies that $\textrm{PGF-dim}_R(F)=\textrm{pd}_R(F)=k$. Using Proposition \ref{three} (1) and the short exact sequence $0\rightarrow M \rightarrow F \rightarrow M \rightarrow 0 $, we get the inequality $k=\textrm{PGF-dim}_R(F)\leq \textrm{PGF-dim}_R(M)$. We conclude that $\textrm{PGF-dim}_R(M)=k$.
\end{proof} 

\begin{prop}\label{dim}
	Let M be an R-module with finite $\textrm{PGF}$-dimension and let $n$ be a nonnegative integer such that $\textrm{PGF-dim}_R(M)\leq n$. Then, $\textrm{Tor}_i^R(I',M)=0$ for every $i>n$ and every module $I'$ with finite injective dimension.
	
\end{prop}

\begin{proof}Since PGF modules are Gorenstein flat, the PGF dimension bounds the Gorenstein flat dimension. The result is known for modules of finite Gorenstein flat dimension (see \cite[Theorem 3.14]{Ho}).\end{proof}

	\begin{theo}\label{theo1}
		Let $M$ be an R-module and $n$ a nonnegative integer. Then, $\textrm{PGF-dim}_R(M)\leq n$ if and only if M is a direct summand of a strongly $n$-PGF module.
	\end{theo}

\begin{proof}
	The case $n=0$ follows from Lemma 4.1 of \cite{SS}. We assume now that $0<\textrm{PGF-dim}_R(M)\leq n$. By Proposition \ref{precover}, there exists a short exact sequence of the form $0\rightarrow M \rightarrow D \rightarrow G^0 \rightarrow 0,$ where $G^0$ is PGF and $\textrm{pd}_R(D)=\textrm{PGF-dim}_R(M)\leq n$. We consider now a truncated projective resolution of $M$ of length $n$, 
$0\rightarrow G_{n-1}\rightarrow P_{n-1}\rightarrow \cdots \rightarrow P_0 \rightarrow M \rightarrow 0,$ where $P_i$ is projective for every $i$ such that $0\leq i \leq n-1$ and $G_{n-1}$ is a PGF module (see \cite[Proposition 2.2]{DE}). Let $G_0=\textrm{Ker}(P_0\rightarrow M)$ and $G_i=\textrm{Ker}(P_{i}\rightarrow P_{i-1})$ for every $i\geq 1$. Then, by Proposition \ref{three} (2) and the short exact sequence $0\rightarrow G_0 \rightarrow P_0 \rightarrow M \rightarrow 0$, we have $\textrm{PGF-dim}_R(G_0)\leq n$. Using again Proposition \ref{three} (2) and the short exact sequences $0\rightarrow G_i \rightarrow P_i \rightarrow G_{i-1}\rightarrow 0$ where $1\leq i \leq n-1$, an inductive argument on $i$ shows that $\textrm{PGF-dim}_R(G_i)\leq n$ for every $i$ such that $0\leq i \leq n-1$. We consider now a projective resolution of $G_{n-1}$, $\cdots \rightarrow P_{n+2} \rightarrow P_{n+1} \rightarrow P_n \rightarrow G_{n-1}\rightarrow 0$, and let $G_i=\textrm{Im}(P_{i+1}\rightarrow P_i)$ for every $i\geq n$. Since $G_{n-1}$ is a PGF module and the class of PGF modules is closed under kernels of epimorphisms, the short exact sequence $0\rightarrow G_n \rightarrow P_n \rightarrow G_{n-1} \rightarrow 0$ implies that $G_n$ is also PGF. Using induction on $i$ and the short exact sequences $0\rightarrow G_i \rightarrow P_i \rightarrow G_{i-1}\rightarrow 0$ for $i\geq n$, the same argument implies that $G_i$ is a PGF-module for every $i\geq n-1$. We conclude that $\textrm{PGF-dim}_R(G_i)\leq n$ for every $i\geq 0$. Since $G^0$ is a PGF-module, by definition it admits a right projective resolution $0\rightarrow G^0\rightarrow P^0 \rightarrow P^1 \rightarrow P^2 \rightarrow \cdots .$ Let $G^i=\textrm{Im}(P^{i-1}\rightarrow P^i)$ for every $i\geq 1$. Then, $\textrm{PGF-dim}_R(G^i)=0$ for every $i\geq 0$. To summarize, we have the following short exact sequences 
\[
\begin{array}{ccccccccc}
	& & \vdots & &\vdots & &\vdots & & \\
	0 & \rightarrow & G^1 & \rightarrow & P^1 & \rightarrow & G^2 &
	\rightarrow & 0 \\
	0 & \rightarrow & G^0 & \rightarrow & P^0 & \rightarrow & G^1 &
	\rightarrow & 0 \\
	0 & \rightarrow & M & \rightarrow & D & \rightarrow & G^0 &
	\rightarrow & 0 \\
	0 & \rightarrow & G_0 & \rightarrow & P_0 & \rightarrow & M &
	\rightarrow & 0 \\
	0 & \rightarrow & G_1 & \rightarrow & P_1 & \rightarrow & G_0 &
	\rightarrow & 0 \\
	& & \vdots & &\vdots & &\vdots & &
\end{array}
\]
\noindent and the direct sum of them yields the short exact sequence $0\rightarrow N \rightarrow Q \rightarrow N \rightarrow 0 $, where $N=\bigoplus_{i\geq 0}G^i \bigoplus M \bigoplus_{j\geq 0}G_j$ and $Q=\bigoplus_{i\geq 0}P^i \bigoplus D \bigoplus_{j\geq 0}P_j$. Then, we obviously have $\textrm{pd}_R(Q)=\textrm{pd}_R(D)\leq n$. Using Proposition \ref{oplus} we get $\textrm{PGF-dim}_R(N)\leq n$. Thus, Proposition \ref{dim} yields $\textrm{Tor}_{n+1}^R(I,N)=0$ for every injective module $I$. We conclude that $N$ is strongly $n$-PGF and $M$ is a direct summand of $N$. 

Conversely, let $M$ be a direct summand of a strongly $n$-PGF module $N$. Then, Proposition \ref{oplus} yields $\textrm{PGF-dim}_R(M)\leq \textrm{PGF-dim}_R(N)$. Since $N$ is strongly $n$-PGF, Proposition \ref{final} (2) implies that $\textrm{PGF-dim}_R(N)\leq n$. We conclude that $\textrm{PGF-dim}_R(M)\leq n$.\end{proof}

\noindent\textbf{Acknowledgments.} Research supported by the Hellenic Foundation for Research and Innovation (H.F.R.I.) under the "1st Call for H.F.R.I. Research Projects to support Faculty members and Researchers and the procurement of high-cost research equipment grant", project number 4226.

\section*{Declarations}
\noindent\textbf{Conflict of interest.} The authors declare that they have no conflict of interest.

\end{document}